\documentclass[12pt]{article}
\usepackage{amsmath}
\usepackage{amsthm}
\usepackage{amsfonts}
\usepackage{amssymb}
\usepackage{fancyhdr}
\usepackage{subfigure}
\usepackage{setspace}
\usepackage{calligra}
\usepackage{epsfig,bm,color}
\usepackage[toc,page]{appendix}
  \newtheorem{theorem}{Theorem}[section]
\newtheorem{df}[theorem]{Definition}

  \newtheorem{lemma}[theorem]{Lemma}

  \newtheorem{corollary}[theorem]{Corollary}
    \newtheorem{conjecture}[theorem]{Conjecture}
\theoremstyle{definition}
 \newtheorem{ex}[theorem]{Example}
   
\newtheorem{remark}[theorem]{Remark}
\newcommand{\ZZ}{\mathbb{Z}}
\newcommand{\NN}{\mathbb{N}}
\newcommand{\ord}{\mathrm{ord}}
\newcommand{\abr}[1]{\langle#1\rangle}

\begin{document}

\title{Strong Skolem Starters}
\author{
\vspace{0.25in}
Oleg Ogandzhanyants\footnote{oleg.ogandzhanyants@mun.ca} \hspace{1cm} Margarita Kondratieva\footnote{mkondra@mun.ca} \hspace{1cm} Nabil Shalaby\footnote{nshalaby@mun.ca}  \\
 Department of Mathematics and Statistics \\
 Memorial University of Newfoundland \\
 St. John's, Newfoundland \\
 CANADA A1C 5S7
}
 \maketitle
\begin{abstract}

This paper concerns a class of combinatorial objects called Skolem starters, and more specifically, strong Skolem starters, which are generated by Skolem sequences. 

 In 1991, Shalaby conjectured that any additive group $\mathbb{Z}_n$, where $n\equiv1$ or $3\pmod{8},\ n\ge11$, admits a strong Skolem starter and constructed these starters of all admissible orders $11\le n\le57$. Only finitely many strong Skolem starters have been known to date.

In this paper, we offer a geometrical interpretation of strong Skolem starters and explicitly construct infinite families of them. 

\end{abstract}

\section{Introduction}\label{introduction}

\indent A \textit{starter} $S$ in an additive abelian group $G$ of odd order $n$ is a partition of the set $G^*$ of all non-zero elements of $G$ into $q=(n-1)/2$ pairs $\{\{s_i,t_i\}\}_{i=1}^q$ such that the elements $\pm(s_i-t_i), i=1,...,q$, comprise $G^*$. Starters exist in any additive abelian group of odd order $n\ge3$. For example, the partition $\{\{x,-x\}\mid x\in G,x\ne0\}$ of $G^*$  is a starter in $G$.

Let $\hat{S}=\{s_i+t_i\mid\{s_i,t_i\}\in S\}$. If $\hat{S}\subset G^*$ and $|\hat{S}|=q$, the starter $S$ is called \textit{strong}.

Strong starters were introduced by Mullin and Stanton in 1968 \cite{b01} for constructing Room squares \cite{b016} and, more generally, Howell designs. Recall that a \textit{Room square} of order $2n$ (or of side $2n-1$) is an $(2n-1) \times (2n-1)$ array filled with $2 n$ different symbols in such a way that:\\
1. Each cell of the array is either empty or contains a 2-subset of the set of the symbols;\\
2. Each symbol occurs exactly once in each column and row of the array;\\
3. Every unordered pair of symbols occurs in exactly one cell of the array.
\smallskip

The question of the existence (or non-existence) of a strong starter in an abelian group is crucial in the theory of Room squares. We refer readers interested in constructions of strong starters to \cite{b014}, \cite{b05}, \cite{b01} and the references therein.

Strong starters in groups of order 3, 5 and 9 do not exist \cite[p.144]{b06}. It is an open question whether there exists a strong starter in every cyclic group of an odd order exceeding 9. In 1981, Dinitz and Stinson \cite{b07} found (by a computer search) strong starters in groups of order $n$ for all odd $7\le n\le999,\ n\ne 9$.

At present, the strongest known general statement on the existence of strong starters is the following \cite[p.625]{b014}: \textit{For any $n>5$ coprime to 6, an abelian group of order n admits a strong starter.}
\smallskip

Skolem starters, which we will be concerned with, are defined only in additive groups $\mathbb{Z}_n$ of integers modulo $n$. We refer to starters in $\ZZ_n$ as starters of order $n$.
\smallskip

A \textit{Skolem sequence} of order $q$ is a sequence $(x_1,x_2,...,x_{2q})$ of length $2q$ in which each number $k\in\{1,...,q\}$ appears exactly twice, so that if $x_i=x_j=k$ then $j-i=k,\ 1\le i<j\le2q$.  Skolem sequences exist iff $q\equiv0$ or $1\pmod{4}$ \cite{b014}. They were originally used by Skolem in 1957 for the construction of Steiner triple systems \cite{b02}.

Given a Skolem sequence $(x_1,x_2,...,x_{2q})$, consider all pairs $\{i_k,j_k\}$ such that $j_k>i_k$ and $x_{i_k}=x_{j_k}=k,\ k=1,...,q$. Clearly, this set of pairs forms a partition of the set $\mathbb{Z}_n^*$ of all non-zero elements of $\ZZ_n$, where $n=2q+1$. Since $j_k-i_k\equiv k\pmod{n}$, (and consequently, $i_k-j_k\equiv -k\pmod{n}),\ k=1,...,q$, this set of pairs is a starter in $\mathbb{Z}_n$. 

\begin{ex}\label{skolem sequence} Sequence $(4,1,1,3,4,2,3,2)$ is a Skolem sequence of order $4$: the length of the sequence is $2\cdot4=8$, and $x_2=x_3=1,x_6=x_8=2,x_4=x_7=3,x_1=x_5=4$, so $1$'s, $2$'s, $3$'s and $4$'s are one, two, three and four positions apart, respectively.  This Skolem sequence yields a starter $T=\{\{2,3\},\{6,8\},\{4,7\},\{1,5\}\}$ in $\mathbb{Z}_9$.
\end{ex}

\begin{df}Let $n=2q+1$, and $1<2<...<2q$ be the order of the non-zero integers modulo n. A starter in $\mathbb{Z}_n$ is Skolem if it can be written as a set of ordered pairs $\{\{s_i,t_i\}\}_{i=1}^q$, where $t_i-s_i\equiv i\pmod{n}$ and $t_i>s_i,\ 1\le i\le q$.
\end{df}
Skolem starters in $\mathbb{Z}_n$ are in one-to-one correspondence with Skolem sequences of order $q=(n-1)/2$. Thus, Skolem starters exist in $\mathbb{Z}_n$ iff $n\equiv1$ or $3\pmod{8}$.

\begin{df}A starter which is both Skolem and strong is called a \textit{strong Skolem starter}.
\end{df}
 \begin{theorem}\label{Shalaby theorem} $($Shalaby, $1991\ \cite[pp.60-62]{b09}.)$ For $11\le n\le 57, n\equiv1$ or $3\pmod{8}$, $\mathbb{Z}_n$ admits a strong Skolem starter.
\end{theorem}
\begin{proof} The Skolem sequences, giving rise to strong Skolem starters of orders $11\le n\le 57,\ n\equiv1$ or $3\pmod{8}$, are presented\footnote{A strong Skolem starter of order 59, the next consecutive order after 57, was constructed by the means described in Lemma \ref{group of units}. We present this starter in the end of the list of those originally found by Shalaby to prove Theorem \ref{Shalaby theorem}.} in Appendix A.
\end{proof}
\begin{conjecture}\label{Shalaby conjecture} $($Shalaby, $1991\ \cite[p. 62]{b09}.)$  
Any $\mathbb{Z}_n,\  n\equiv 1$ or $3\pmod{8},\ n\ge11,$ admits a strong  Skolem starter.
\end{conjecture}
The value of strong Skolem starters of order $2q+1$ is in their applicability in constructing Room squares of order $2q+2$  on one hand, and \textit{Steiner triple systems}, STS($6q+1$), on the other. Recall that an STS($v$) is a collection of $3$-subsets, called $blocks$, of a $v$-set $S$, such that every two elements of $S$ occur together in exactly one of the blocks.

Example \ref{RS+STS construction} illustrates the use of strong Skolem starters.

\begin{ex}\label{RS+STS construction}
Let $S=\{\{s_i,t_i\}\}_{i=1}^{5}=\{\{1,2\},\{7,9\},\{3,6\},\{4,8\},\{5,10\}\}$. It is a strong Skolem starter in $\ZZ_{11}:\ \hat{S}=\{1+2,7+9,3+6,4+8,5+10\pmod{11}\}=\{3,5,9,1,4\}\subset\ZZ_{11}^*,\ |\hat{S}|=5$.
\begin{itemize}
\item A Room square of order $12$ constructed out of $S$ $($its pairs appear in the first row$)$:

\begin{tabular}{|c|c|c|c|c|c|c|c|c|c|c|}
    \hline
    $\infty,0$ &4,8&-&1,2&5,10&7,9& -&-&- &3,6&- \\ 
\hline
 -&$\infty,1$&5,9&- &2,3&6,0 &8,10&- & -&-&4,7 \\ 
\hline 
5,8 &- & $\infty,2$ & 6,10 &- & 3,4 & 7,1 & 9,0& -& - &-\\
    \hline
-&6,9&-&$\infty,3$&7,0&-&4,5&8,2 &10,1 & -& - \\
\hline
-&-&7,10&-&$\infty,4$&8,1&- &5,6& 9,3& 0,2&-  \\
\hline
-&-&-&8,0&-&$\infty,5$&9,2&- &6,7 &10,4& 1,3  \\
\hline
2,4&-&-&-&9,1&-&$\infty,6$& 10,3&- &7,8&0,5\\
\hline
1,6&3,5&-&-&-&10,2&-&$\infty,7$& 0,4&- & 8,9\\
\hline
9,10&2,7&4,6&-&-&-&0,3&-&$\infty,8$& 1,5&- \\
\hline
-&10,0&3,8&5,7&-&-&-&1,4&-&$\infty,9$& 2,6\\
\hline
 3,7&-&0,1&4,9&6,8&-&-&-&2,5&-&$\infty,10$ \\
\hline
  \end{tabular}
\item The translates of the base blocks $\{0,i,t_i+5\},\ 1\le i\le5$, obtained from $S$, generate  an STS$(31):$
\begin{center}
\begin{tabular}{cccccc}
$\{0,1,7\}$&$\{0,2,14\}$&$\{0,3,11\}$&$\{ 0,4,13\}$&$\{0,5,15\}$&$\pmod{31}$.\\
\end{tabular}
\end{center}
\end{itemize}
\end{ex}
In Section \ref{geometry}, we give a geometrical interpretation of strong Skolem starters. In Section \ref{cardioidal starters intro}, we introduce an important subset of Skolem starters called cardioidal starters, and prove some statements about their basic properties. In Section \ref{constructing}, we find the necessary and sufficient conditions for the existence of a cardioidal starter of order $n$ (Theorem \ref{major theorem}), present the way to construct cardioidal starters (Lemma \ref{group of units} and Theorem \ref{major theorem}), and find the necessary and sufficient conditions for the existence of strong cardioidal starters (Theorem \ref{final theorem}). Explicit construction of infinite families of strong cardioidal (and hence, strong Skolem) starters supports Conjecture \ref{Shalaby conjecture} and is the main result of the paper. In Section \ref{order divisible by 3}, we briefly discuss cardioidal starters which are not strong, and articulate one of possible ways of using a strong cardioidal starter of order $n$ to construct strong Skolem starters of order $3n$.
\section{Geometrical interpretation of strong Skolem starters}\label{geometry}

J. Dinitz, in his PhD thesis, proposed a geometrical interpretation \cite {b07} of a strong starter in $\mathbb{Z}_n$ (see Figure \ref{fig:Dinitz}). The following theorem is based on that interpretation:

\begin{theorem}\label{Dinitz theorem}Let $q=(n-1)/2$ and $S=\{x_i,y_i\}_{i=1}^q$ be a strong starter in $\ZZ_n$. Label $n$ equally spaced points on a circle by the elements of $\mathbb{Z}_n\  (cyclically)$. If $\{x,y\}\in S$, then join points $x$ and $ y$ on the circle by a straight line $($a chord$)$. The set $C=\{[x_i,y_i]\}_{i=1}^q$ of $q$ chords thus formed will have the following\footnote{However only  properties $(\bf{a})$--$(\bf{c})$ were mentioned in \cite{b07}, the set of pairs corresponding to a collection of chords not satisfying $(\bf{e})$ is not a starter. Also, a set of chords satisfying $(\bf{a})$--$(\bf{c})$ and $(\bf{e})$ but violating $(\bf{d})$  is not strong. The starter $\{1,2\}$ in $\ZZ_3$ is a clear example of such a starter. A less trivial one is the following starter in $\mathbb{Z}_{21}$: $\{18,19\},\{1,3\},\{20,2\},\{10,14\},\{8,13\},\{6,12\},\{4,11\},\{9,17\},\{7,16\},\{5,15\}$. Here the chord $[8,13]$ violates $(\bf{d})$. } properties:\\
$(\bf{a})$ no two chords have the same length;\\
$(\bf{b})$ no two chords are parallel;\\
$(\bf{c})$ no two chords share a common endpoint;\\
$(\bf{d})$ the perpendicular bisector of any chord does not pass through point $0$ on the circle\footnote{ In particular, property $(\bf{d})$ of Theorem \ref{Dinitz theorem} means that no chord in Figure \ref{fig:Dinitz} is horizontal.};\\
$(\bf{e})$ point $0$ is not an endpoint of any chord.\\
Conversely, any such a geometric configuration of q chords generates a strong starter in $\mathbb{Z}_n$.
\end{theorem}
\begin{proof}
 Let $S=\{x_i,y_i\}_{i=1}^q$ be a strong starter in $\ZZ_n$. Since $S$ is a starter, for any $1\le i<j\le q$, we have: $x_i-y_i\not\equiv\pm(x_j-y_j)\pmod{n}$. Hence, property $(\bf{a})$ holds. Also, $\cup_{i=1}^q\{x_i,y_i\}=\ZZ^*_n$. Hence, properties $(\bf{c})$ and $(\bf{e})$ hold. Since $S$ is strong, for any $1\le i<j\le q$, we have: $x_i+y_i\not\equiv x_j+y_j\pmod{n}$. Then $x_i-x_j\not\equiv -(y_i-y_j)\pmod{n}$. That means that the midpoints of the arcs $\stackrel{\frown}{x_iy_i}$ and $\stackrel {\frown}{x_jy_j}$ do not lie on the same diameter of the circle. Hence, property $(\bf{b})$ holds. Finally, since $S$ is strong, $x_i+y_i\not\equiv0\pmod{n},\ 1\le i\le q$. Therefore, property $(\bf{d})$ holds. That proves the direct statement.
\smallskip

Now, consider a set $C$ of $q$ chords connecting $n$ equally spaced points on a circle, and satisfying properties $(\bf{a})$--$(\bf{e})$. Then there is the only point, 0, which is not an endpoint of any of the $q$ chords. Label the $n-1$ remaining points cyclically by elements of $\ZZ_n$. Let $S$ be the following partition of $\ZZ^*_n:\ S=\{\{x,y\}|\ [x,y]\in C\}$. By reversing the chain of reasonings shown in part 1, we conclude that properties $(\bf{a})$, $(\bf{c})$ and $(\bf{e})$ imply that $S$ is a starter in $\mathbb{Z}_n$. Properties $(\bf{b})$ and $(\bf{d})$ imply that $S$ is strong. 
\end{proof}

\begin{figure}[htbp]
  \centering
  \includegraphics[width=0.6\textwidth,trim={5cm 4.8cm 8.5cm 16.3cm},clip]{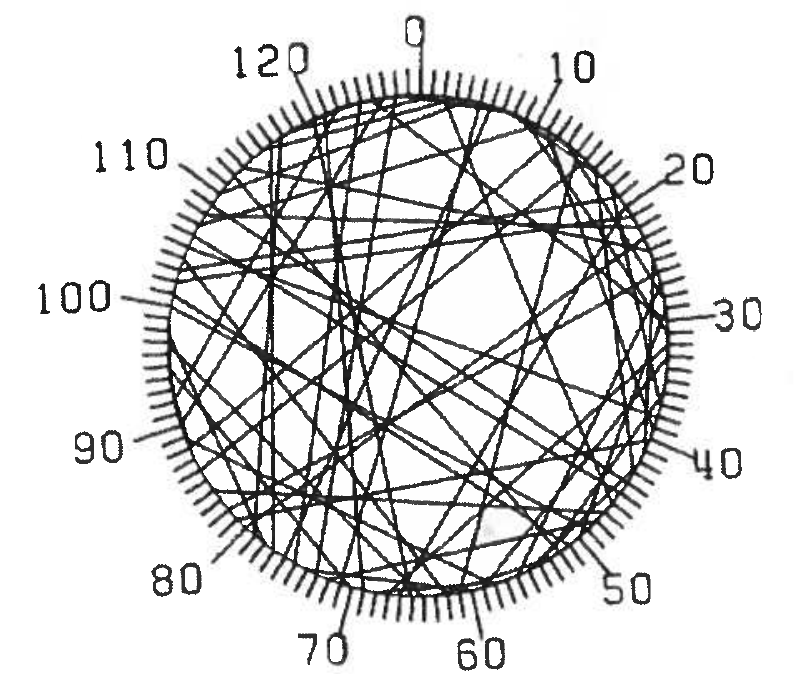}
 \caption{A strong starter of order 129, as it was interpreted by Dinitz.}
\label{fig:Dinitz}
\end{figure}

 Consider the same circle with $n$ labelled points as described above. Let us call the segment between the point 0 and the center of the circle the 0-\textit{radius}. Then impose one more restriction on the set of chords in addition to $(\bf{a})$--$(\bf{e})$:\\
\textit{$(\bf{f})$ no chord of the set $C$ crosses the $0$-radius.}

\begin{theorem}\label{conditions} Let S be a strong Skolem strater in $\mathbb{Z}_n$ and let $C$ be the set of chords representing $S$ as described in Theorem \ref{Dinitz theorem}. Then the set $C$ satisfies properties $(\bf{a})$--$(\bf{e})$ from Theorem \ref{Dinitz theorem} along with property $(\bf{f})$.

Conversely, any such a geometric configuration of $(n-1)/2$ chords generates a strong Skolem starter in $\mathbb{Z}_n$.
\end{theorem}
\begin{proof}
 Let us order the elements of $\mathbb{Z}_n^*$: $1<...<n-1=2q$. Since $S$ is Skolem, then for any $\{x,y\}\in S$, given $y>x$, there holds $y-x\le q$. Hence, property $(\bf{f})$ holds. Conversely, by Theorem \ref{Dinitz theorem}, a set of chords satisfying properties $(\bf{a})$--$(\bf{e})$, generates a strong starter $\{\{s_i,t_i\}\mid t_i-s_i\equiv i\pmod{n},\ 1\le i\le q\}$. By $(\bf{f})$, no chord crosses the 0-radius. Therefore, $\ t_i>s_i,\ 1\le i\le q$. Hence, the strong starter is Skolem.
\end{proof}
\begin{remark}

Properties $(\bf{a}),(c),(e)$ and $(\bf{f})$ present a geometrical interpretaion of a Skolem starter. In addition, properties $(\bf{b})$ and $(\bf{d})$ imply that the starter is strong Skolem.
\end{remark}

\section{Cardioidal starters as a subset of Skolem starters}\label{cardioidal starters intro}\label{cardioidal subset}
Property $(\bf{f})$ in Theorem \ref{conditions} motivates us to seek Skolem starters of a special type to ensure that the correspoding chords do not cross the 0-radius. One possible approach is to pick chords, satisfying properties $(\bf{a})$--$(\bf{e})$, out of a bigger collection of chords tangent to a curve whose tangents never cross the 0-radius. 

Let us consider a unit circle $x^2+y^2=1$ in a Cartesian coordinate plane with a point labelled 0 at $(x,y)=(0,1)$. Let us associate every point on the circle with the clockwise arclength $\theta\pmod{2\pi}$ between the point 0 and this point. 

\begin{lemma}\label{caustic} The envelope of the family of chords $[\theta,2\theta\pmod{2\pi}]$ in the circle described above is a cardioid\footnote{More generally, the envelope of chords $[\theta,a\theta],\ a\in\mathbb{N}\setminus\{1\}$ is an epicycloid with $a-1$ cusps \cite{b012}.} with vertex at $(0,1)$ and cusp at $(0,-1/3)$.
\end{lemma}

Lemma \ref{caustic} is well-known \cite{b012}, \cite[p.207]{b08} from geometry. It also expresses a fact of optics: a light beam, going from 0 inside the circular mirror, reaches some point $\theta$  and then arrives to point $2\theta\pmod{2\pi}$. The envelope of these first reflection chords, that is, the chords of type $[\theta,2\theta\pmod{2\pi}]$, is the cardioid. 

\bigskip

In the sequel, we will treat  integers modulo $n$ as a ring, keeping the notation $\ZZ_n$ used previously for the corresponding additive group. The notation $\ZZ_n^*$ is saved for the set of all non-zero elements of $\ZZ_n$.

Given an odd $n\ge3$, let us label points $\theta_k=2k\pi/n, k=1,...,n-1$, by $k$, respectively. For each $k\in\ZZ_n^*$, draw the chord $[k, 2k\pmod{n}]$. (See Figure  \ref{envelope}.)
\begin{lemma}\label{family} For an odd $n\ge3$, the family of chords $\{[i,2i\pmod{n}]\}_{i=1}^{n-1}$ satisfies properties $(\bf{e})$ and $(\bf{f})$, and moreover, if $n\not\equiv 0\pmod {3}$ then properties $(\bf{b}),\ (d)$ are satisfied as well.
\end{lemma}
\begin{proof} Property $(\bf{e})$ holds trivially. Property $(\bf{f})$ holds by Lemma \ref{caustic}, since the chords are tangent to the cardioid and $n$ is odd. Now, let $3\nmid
n$:\\
1) $\forall\  i,j\in \mathbb{Z}_n: i+2i\equiv j+2j\pmod{n}$ is equivalent to $i=j$, and hence, property $(\bf{b})$ holds;\\
2) $\forall\  i\in \mathbb{Z}^*_n: i+2i\not\equiv 0\pmod{n}$, and hence, property $(\bf{d})$ holds.
\end{proof}
\begin{figure}[htbp]
  \centering
\subfigure[]{  
\includegraphics[ width=0.47\textwidth,trim={1.8cm 5.4cm 2.5cm 5.4cm},clip]{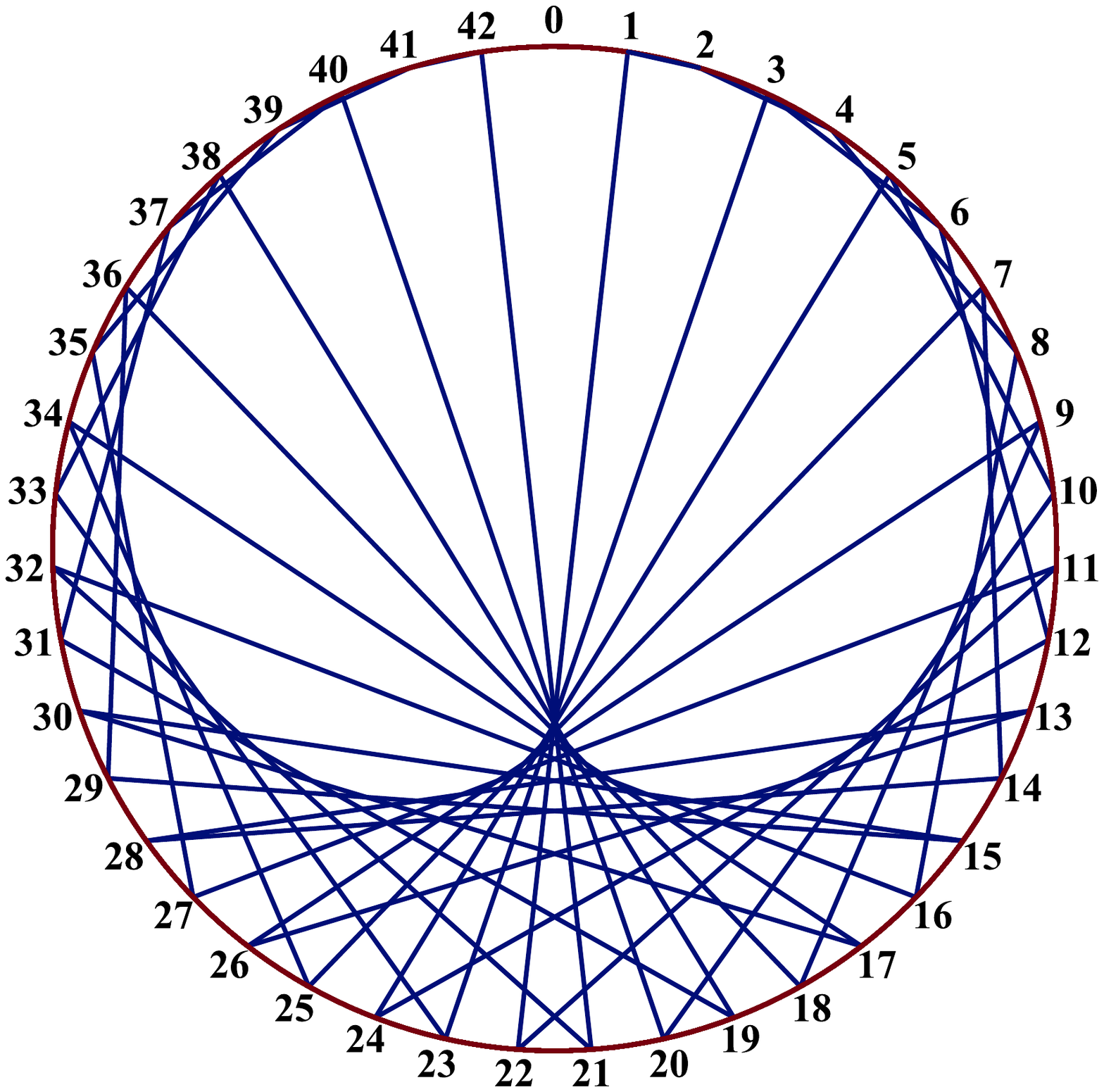}

\label{envelope}
 }
\subfigure[]{  
 \includegraphics[width=0.47\textwidth,trim={1.8cm 5.4cm 2.5cm 5.4cm},clip]{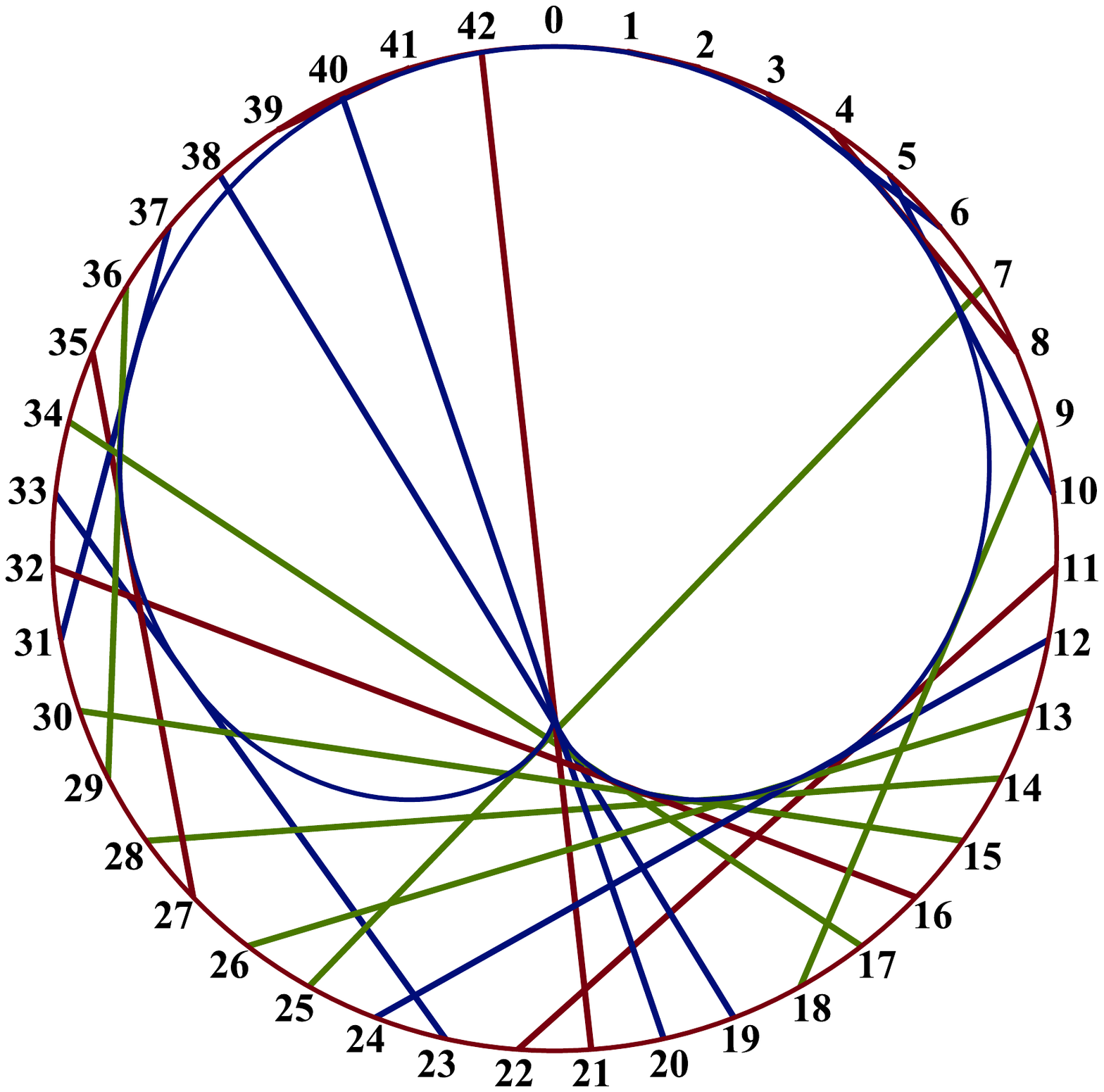}
  
\label{starter 43}
   } 
\caption{(a) The 42 circular chords of type $[i,2i\pmod{43}],\ i\in\mathbb{Z}^*_{43}$, are all tangent to the cardioid. (b) A geometrical interpretation of a strong cardioidal starter of order 43. We see one half of the family of chords shown on the left. The other half generates another strong cardioidal starter of order 43. }
\end{figure}

\begin{df}We will call a chord of type $[i,2i\pmod{n}],\ i\in\ZZ_n^*$, cardioidal$\pmod{n}$. A starter of order $n$ is cardioidal if all of its corresponding chords are cardioidal$\pmod{n}$.
\end{df}
\begin{corollary}\label{subset}Cardioidal starters comprise a nonempty proper subset of the set of Skolem starters.
\end{corollary}
\begin{proof}
The starter $S$ presented in Example \ref{RS+STS construction} is a cardioidal starter of order 11. Hence the set of cardioidal starters is not empty.
By Lemma \ref{family}, the corresponding chords of a cardioidal starter satisfy property $(\bf{f})$. Hence cardioidal starters are Skolem. The converse is not true, for instance, the starter $T$  in Example \ref{skolem sequence} is a Skolem starter of order 9 which is not cardioidal as it contains no pair of type $\{i,2i\pmod{9}\}$.
\end{proof}

\begin{lemma}\label{basic statement} Let S be a cardioidal starter of order $n$. Then S is strong iff $3\nmid n$.
\end{lemma}

\noindent\textit{Proof.} Let $C$ be the set of chords corresponding to $S$. Let also $C'$ be the family of all cardioidal chords (mod $n$).
\begin{enumerate}
\item Case $3\nmid n$. Since $S$ is a starter, the set $C$ satisfies properties $(\bf{a}),(\bf{c}),(\bf{e})$. Since $C\subset C'$ and $3\nmid n$, then, by  Lemma \ref{family}, $C$ satisfies properties $(\bf{b}),(\bf{d})$. Hence, by Theorem \ref{Dinitz theorem}, the starter $S$ generated by $C$ is strong.

\item Case $3\mid n$. Let $n=3l,\ l\in\NN$. The family $C'$ contains the chords of $(3l-1)/2$ different sizes. Let $Q\in C'$ be the chord $[l,\ 2l]=[-l,\ -2l\pmod{n}]$. Hence, $Q$ is the only chord of its size in $C'$. Further, since $C\subset C'$ generates a starter, $C$ satisfies property $(\bf{a})$ and consists of $(3l-1)/2$ chords. Hence, $Q\in C$. But then $C$ violates property $(\bf{d})$, since $Q$ is horizontal. This implies that the starter $S$ is not strong. \hfill $\Box$
\end{enumerate}

A strong starter is called \textit{skew} if the sums of its pairs $s_i+t_i$, and their negatives $-s_i-t_i$ are all distinct \cite{b014}. In other words, a starter of order $n,\ \{\{s_i,t_i\}\}_{i=1}^{(n-1)/2}$, is skew if $\{\pm(s_i+t_i)\}=\mathbb{Z}_n^*$. Skew starters are important for constructing \textit{skew Room squares}\footnote{A Room square is called skew if, whenever $i\ne j$, precisely one of the $(i,j)$th and $(j,i)$th cells is occupied \cite{b014}. An example of a skew Room square constructed by use of a cardioidal starter is presented in Section \ref{introduction}. }.

\begin{theorem}\label{skew} A cardioidal starter of order $n,\ 3\nmid n$, is skew.
\end{theorem}
\begin{proof}
If a starter of order $n$ is cardioidal and $3\nmid n$, then, by Lemma \ref{family}, for any pair of its corresponding chords $[i,2i\pmod{n}]$ and $[j,2j\pmod{n}]$, there holds $3i\not\equiv\pm 3j$ (mod $n$). Therefore, $\{\{\pm(i+2i)\pmod{n}\}\}_{i=1}^{(n-1)/2}=\mathbb{Z}^*_n$. Hence, the starter is skew.
\end{proof}
\section{Construction of cardioidal starters}\label{constructing}

By $G_n$, we denote the group of  units of the ring $\ZZ_n$ (elements invertible with respect to multiplication).
We denote by $\abr{x}_n$ the cyclic subgroup of $G_n$ generated by $x\in G_n$.

Whenever the group operation is irrelevant, we will consider $G_n$ and its cyclic multiplicative subgroups $\abr{x}_n$ in the set-theoretical sense and denote them by $\underline{G}_n$ and $\abr{\underline{x}}_n$ respectively. Also, we will use notation $aB=\{ab|\ b\in B\}$, where $a\in\ZZ$ and $B\subset\ZZ$.

Let us denote by $\ord_n (x)$ the order of an element $x$ in the group $G_n\ (\ord_n (x)=|\abr{x}_n|$), and introduce two classes of numbers 
\begin{equation}\label{C2 class}
C_2=\{p\ {\rm{prime}}\mid\ \ord_p (2)\equiv 2\pmod{4}\},
\end{equation}
and a subset of $C_2$,
\begin{equation}\label{P2 class}
P_2=\{ p\in C_2\mid\ p\equiv1\pmod{8}\}.
\end{equation}

Before we proceed to the statement about the existence of cardioidal starters of certain orders, we prove three auxilliary lemmas.

\begin{lemma} A prime $p\in C_2$ iff either $p\in P_2$ or $p\equiv 3\pmod{8}$.
\end{lemma}
\noindent\textit{Proof.} Let $p$ be a prime number. Then $G_p$ is cyclic.  Consider four possible cases:
\begin{enumerate}
\item $p\equiv1\pmod{8}$. If $p\in P_2$ then $p\in C_2$ by definition of $P_2$. Otherwise, $p\notin C_2$.
\item $p\equiv3\pmod{8}$. By a special case of the Quadratic Resiprocity Law \cite{b015}, 2 is a quadratic non-residue modulo $p$. Hence, $\ord_p(2)$ is even. Also, $\ord_p (2)$ divides $|G_p|=p-1\equiv2\pmod{8}$, hence, $\ord_p(2 )\not\equiv0\pmod{4}$. Therefore, $p\in C_2$. 
\item $p\equiv5\pmod{8}$. In that case, 2 is a quadratic non-residue modulo $p$ \cite{b015}. Hence, $\ord_p(2)=\mu$ is even. Then $2^{\mu/2}$ is a square root of one, which is not 1 (otherwise, $\ord_p(2)=\mu/2<\mu$). Since $p$ is prime, the only square roots of 1 modulo $p$ are $\pm1$, hence, $2^{\mu/2}\equiv -1\pmod{p}$. Since $-1$ is a quadratic residue modulo $p$ \cite{b015}, $\mu/2$ is even. Therefore, $\mu\equiv0\pmod{4}$. This implies $p\notin C_2$.
\item $p\equiv7\pmod{8}$. In that case, 2 is a quadratic residue modulo $p$ \cite{b015}. Hence, $\ord_p(2)$ must divide $|G_p|/2=(p-1)/2$. Since the latter is odd, so is $\ord_p(2)$. Hence, $p\notin C_2$.\hfill $\Box$
\end{enumerate}
\begin{remark}
The class $C_2\setminus P_2$ is infinite by Dirichlet's theorem on primes in arithmetic progressions. The class $P_2=\{281,617, 1033,1049,1097,1193,1481,1553,1753,1777,...\}$ is also infinite. It follows, for example, from the infinitude of primes congruent to $9\pmod{16}$ such that $\ord_p(2)\equiv2\pmod{4}$ \cite{b019}. 
\end{remark}

Denote by $\overline{C_2}$ the \textit{multiplicative closure}\footnote{Recall that a non-empty subset $X$ of a ring $R$ is multiplicatively closed if whenever $u,v\in X$, then $uv\in X$. The multiplicative closure of a set $Y\subset R$, denoted by $\overline{Y}$, is the smallest multiplicatively closed set containing $Y$.} of $C_2$.  So if $n\in\overline{C_2}$, then $n$ is a product of $($not necessarily distinct$)$ primes from $C_2$.

\begin{lemma}\label{subgroup of 2} Let $n\in\overline{C_2}$. Then in $G_n,$ the element $-1$ is an odd power of $2$ and $\ord_n (2)\equiv2\pmod{4}$. 
\end{lemma}
\noindent\textit{Proof.} Consider two cases:
\begin{enumerate}
\item $n=p^k,\ p\in C_2, k\in\mathbb{N}$. 

Suppose, by way of contradiction, that $\mu=\ord_n 2$ is odd. But $2^{\mu}\equiv1\pmod{n}$ implies that $2^{\mu}\equiv1\pmod{p}$, which contradicts the fact that $\ord_p (2)$ is even, since $p\in C_2$. Hence, $\mu$ is even. 

Since $\mu$ is even, we can write $(2^{\mu/2})^2\equiv1 \pmod{n}$. Since $G_n$ is cyclic \cite{b015} and $|G_n|=p^{k-1}(p-1)$ is even, the congruence $x^2\equiv 1$ has only two solutions in $G_n$. These are $\pm 1$. Therefore, $2^{\mu/2}\equiv1$ or $2^{\mu/2}\equiv-1\pmod{n}$. But $2^{\mu/2}\equiv 1\pmod{n}$ is impossible since $\mu=\ord_n(2)>\mu/2.$ We conclude that $2^{\mu/2}\equiv-1\pmod{n}$. 

Let us show $\lambda=\mu/2$ is odd.  Clearly, $2^{\lambda}\equiv-1\pmod{n}$ implies $2^{\lambda}\equiv-1\pmod{p}$. Then consider $\abr{2}_p$. Since $p\in C_2,\ \nu=\ord_p(2) /2$ is odd. In $G_p,\ 2^{\nu}$ is a square root of 1, which is distinct from 1, that is, $2^{\nu}\equiv -1\pmod{p}$. Therefore, $\lambda$ is odd, that is, $-1$ is an odd power of 2 modulo $n$, and hence, $\ord_n(2)\equiv2\pmod{4}$.

\item $n=\Pi_{i=1}^m n_i$, $m>1$, where $n_i=p_i^{k_i},\ p_i\in C_2,\ i=1,..., m$, are pairwise distinct, $k_i\in\NN,\ i=1,...,m$. 

Denote $t_i=\ord_{n_i}(2)/2$ and $P=\Pi_{i=1}^m t_i$. By Case 1, we have $2^{t_i}\equiv-1 \pmod{ n_i},\ i=1,...,m$. 
Raising these congruences to the power $P/t_i$, respectively, we obtain $m$ congruences: $2^P\equiv-1 \pmod{ n_i}$,  where the numbers $n_i$ are pairwise coprime, $i=1,...,m$.  The system of congruences $x\equiv-1 \pmod{ n_i},\ i=1,...,m$, for an unknown element $x$ of the ring $\ZZ_n$, has an obvious solution $x\equiv-1\pmod{n}$. 
By the Chinese Reminder Theorem, this solution is unique. We conclude that $2^P\equiv-1 \pmod{ n}$. But $P$ is odd, since, by Case 1, each $t_i,\ i=1,...,m$, is odd. Hence, $-1$ is an odd power of 2. It implies that $\ord_n (2)\equiv2\pmod{4}$.\hfill $\Box$
\end{enumerate}
Let $A$ be a set of $k$ pairs of integers modulo $n:\ A=\{\{x_i,y_i\}\pmod{n}\}_{i=1}^k,\ k\in\mathbb{N}$. We define
\begin{equation}
\Delta A=\{\pm(x_i-y_i)\pmod{n}\}_{i=1}^k
\end{equation}
to be the collection of the differences between the two elements within each pair of $A$. We will say that $\Delta A$ is generated by $A$.
\begin{corollary}\label{the differences}Let $\mu=\ord_n(2)$, where $n\in\overline{C_2}$. Then the following is a partition of the set $\langle\underline{2}\rangle_n:\ A=\{\{2^{2i},2^{2i+1}\}\}_{i=0}^{(\mu-2)/2}$. Consequently, $\Delta A=\langle \underline{2}\rangle_n$.
\end{corollary}
\begin{proof} By Lemma \ref{subgroup of 2}, the number $\mu=\ord_n (2)$ is even, hence, the partition $A$, as indicated, is possible. Also by Lemma \ref{subgroup of 2}, $-1$ is an odd power of $2$, that is, $-1\in\langle \underline{2}\rangle_n\setminus\langle \underline{4}\rangle_n$. Then we have $\Delta A=\{\pm(2^1-2^0),...,\pm(2^{\mu-1}-2^{\mu-2})\}=\{\pm 2^{2i}\}_{i=0}^{(\mu-2)/2}=\langle \underline{4}\rangle_n\cup -1\cdot\langle \underline{4}\rangle_n=\langle \underline{2}\rangle_n$.
\end{proof}

\begin{df}For an odd $n\in\mathbb{N}\setminus\{1\}$ and $i\in\ZZ^*_n$, let us call the pair $\{i,2i\pmod{n}\}$ \textit{a cardioidal pair} in $\ZZ_n^*$. A partition $A$ of a set $X\subset\ZZ^*_n$ into cardioidal pairs is called a \textit{cardioidal partition} of $X$ in $\ZZ_n^*$.
\end{df}
\begin{lemma}\label{group of units}The set $\underline{G}_n,\ n\in\overline{C_2}$,  admits a cardioidal partition $A$ in $\ZZ_n^*$ such that $\Delta A=\underline{G}_n$.
\end{lemma}
\begin{proof}

Since $\langle 2\rangle_n$ is a subgroup of $G_n,\ \ord_n(2)$ divides $|G_n|$. Let $\mu=\ord_n(2)$ and $ b=|G_n|/\mu$. Denote the $b-1$ cosets of $\langle 2\rangle_n$ by $B_1,B_2,...,B_{b-1}$.

Let  $\beta_1\in B_1,...,\beta_{b-1}\in B_{b-1}$ and $\beta_0=1$. Then $A$ is a cardioidal partition of $\underline{G}_n$ in $\ZZ_n^*$: 
\begin{equation}\begin{split}
A&=\{\{2^0,2^1\},...,\{2^{\mu-2},2^{\mu-1}\},\{\beta_1 2^0,\beta_12^1\},...,
 \{\beta_{b-1}2^{\mu-2},\beta_{b-1}2^{\mu-1}\}\}\\
&=\{\{\beta_i 2^{2j},\beta_i2^{2j+1}\}\mid 0\le i\le b-1,0\le j\le (\mu-2)/2\}.
\end{split}\end{equation}
Let $i\in G_n$. Since $-1\in\abr2_n$, the four elements, $\pm i,\pm 2i\pmod{n}$ belong to the same coset of $\abr 2_n$ in $G_n$. So, there is $j,\ 0\le j\le b-1$, such that $i\equiv \beta_j2^y$ and $-i\equiv \beta_j2^z$.  Since $-1$ is an odd power of 2 modulo $n$, the parities of $y$ and $z$ are different. Hence, either $\{i,2i\}\in A$ or $\{-i,-2i\pmod{n}\}\in A$. In either case, $\Delta A$ contains $i$, since either $2i-i\equiv i\pmod{n}$ or $-i-(-2i)\equiv i\pmod{n}$. Hence, $\Delta A$ contains all elements from $\underline{G}_n$, that is, $G_n\subset \Delta A$. On the other hand, $|\Delta A|\le |\underline{G}_n|$. Hence, $\Delta A=\underline{G}_n$.
\end{proof}

\begin{corollary}\label{one prime}For any  prime $n\in C_2$, there exists a cardioidal starter of order $n$.
\end{corollary}
\begin{proof}
By Lemma \ref{group of units}, $\underline{G}_n$ admits a cardioidal partition $A$ in $\ZZ_n^*$, such that $\Delta A$ comprises $\underline{G}_n$. Since $n$ is prime, $\underline{G}_n=\ZZ^*_n$. Hence, $A$ is a cardioidal starter of order $n$.
\end{proof}

\begin{ex}\label{starter 19,43}
\begin{enumerate}  Exapmles of constructing cardioidal starters of prime orders.
\item $n=19$.

As $2$ is a primitive root in $G_{19}$, we get a strong Skolem starter of order $19$ right away:\\
$\{\{1,2\},\{4,8\},\{16,13\},\{7,14\},\{9,18\},\{17,15\},\{11,3\},\{6,12\},\{5,10\}\}$,\\
and the corresponding Skolem sequence is 
$(1,1,8,4,5,6,7,4,9,5,8,6,3,7,2,3,2,9)$.
\item $n=43$. $($see Figure \ref {starter 43}$)$

Here, $2$ is not a primitive root in $G_{43}$, so we construct a starter by taking the pairs from $\langle 2\rangle_{43}$, $(|\langle 2\rangle_{43}|=14)$,  and from the two its cosets: $\ZZ^*_{43}=\langle \underline{2}\rangle_{43}\cup3\langle \underline{2}\rangle_{43}\cup13\langle \underline{2}\rangle_{43}$ .\\
The pairs from $\langle \underline{2}\rangle_{43}:\ \{1,2\},\{4,8\},\{16,32\},\{21,42\},\{41,39\},\{35,27\},\{11,22\}$;\\
the pairs from $3\langle \underline{2}\rangle_{43}:\ \{3,6\},\{12,24\},\{5,10\},\{20,40\},\{37,31\},\{19,38\},\{33,23\}$;\\
the pairs from $13\langle \underline{2}\rangle_{43}:\ \{13,26\},\{9,18\},\{36,29\},\{15,30\},\{17,34\},\{25,7\},\{14,28\}$.

The corresponding Skolem sequence is:  $(1,1,3,4,5,3,18,4,9,5,11,12,13,14,15,16,17,\\9,19,20,21,11,10,12,18,13,8,14,7,15,6,16,10,17,8,7,6,19,2,20,2,21)$.
\item $n=281$. This is the smallest $n\in P_2$. See the list of the pairs of a cardioidal starter of order $281$ in Appendix \rm{B}.
\end{enumerate}
\end{ex}

The following theorem gives the necessary and sufficient conditions of the existence of a cardioidal starter of order $n$ and presents an effective way to construct it using Lemma \ref{group of units}.
\begin{theorem}\label{major theorem} A cardioidal starter of order $n$ exists iff $n\in\overline{C_2}$.

\end{theorem}
\begin{proof} Let $n\in\overline{C_2}$. Then $n=\prod_{i=1}^m p_i^{k_i}$, where $p_i\in C_2,\ i=1,...,m$, are pairwise distinct, and $k_i\in\NN,\ i=1,...,m$.

Let $B=\{b_i\}_{i=0}^t$ be the set of the divisors of $n$, where $1=b_0<b_1<...<b_t=n$ and $t=\prod_{i=1}^m (k_i+1)-1$. Consider the sets $a_i\underline{G}_{b_i}=\{x\in\ZZ_n^*\mid\gcd(x,n)=a_i\}$, where $a_ib_i=n,\ i=1,...,t$. Since every element $x\in\mathbb{Z}^*_n$ lies in one and only one of these sets, the collection $\{a_i\underline{G}_{b_i}\}_{i=1}^t$ forms a partition of $\mathbb{Z}_n^*$ into $t$ subsets. 

By Lemma \ref{group of units}, every set $\underline{G}_{b_i}$ admits a cardioidal partition in $\ZZ^*_{b_i},\ i=1,...,t$. Hence, every set $a_i\underline{G}_{b_i}$ admits a cardioidal partition $A_i$ in $\ZZ^*_n,\ i=1,...,t$. By Lemma \ref{group of units}, each $A_i$ generates the set $\Delta A_i=a_i\underline{G}_{b_i},\ i=1,...,t$, so that $\bigcup_{i=1}^t \Delta A_i=\ZZ_n^*$. Hence, $\bigcup_{i=1}^t  A_i$ is a cardioidal starter of order $n$. 

\smallskip

To prove the converse, let a prime $p\notin C_2$ divide $n$. Then denote $N_p=n/p$ and consider the set $V=N_p\ZZ^*_p\subset \ZZ^*_n$.  Clearly, no pair of type $\{i,2i\pmod{n}\},\ i\notin V$, generates a difference which is an element of $V$. Hence, if there exists a cardioidal starter of order $n$ then $V$ admits a cardioidal partition $A$  in $\ZZ_n^*$ such that $\Delta A=V$. Equivalently, the existence of a cardioidal starter of order $n$ implies the existence of a cardioidal starter of order $p$. Consider two cases:
\begin{enumerate}
\item $p\equiv5$ or $7\pmod{8}$. A cardioidal starter of order $p$ gives rise to a Skolem sequence of order $(p-1)/2\equiv 2$ or $3\pmod{4}$ which is impossible \cite{b09}.
\item  $p\equiv1\pmod{8}$. Suppose, by way of contradiction, that there exists a cardioidal starter of order $p$. If a pair $\{x,y\}$ is cardioidal  in $\ZZ_p^*$, then $x$ and $y$ belong to the same coset of $\langle \underline{2}\rangle_p$. Hence, the existence of a cardioidal partition of $\ZZ^*_p$ implies a cardioidal partition of $\abr 2_p$  in $\ZZ_p^*$. In the case of odd $\ord_p(2)$, it is impossible. Then suppose that $\ord_p(2)$ is even. Since $p\notin C_2$ then $\ord_p(2)\equiv0\pmod{4}$. In this case, $-1\in\langle \underline{2}\rangle_p$, and hence, no element of $\langle \underline{2}\rangle_p$ can be a difference of a pair $\{i,2i\pmod{p}\},\ i\notin \langle \underline{2}\rangle_p$. Hence, the cardioidal starter of order $p$ must contain such a partition $A$ of $\langle \underline{2}\rangle_p$ that $\Delta A=\langle \underline{2}\rangle_p$.

The only two possible cardioidal partitions of $\langle\underline{2}\rangle_p$  in $\ZZ_p^*$ are: $A_1=\{\{2^k,2^{k+1}\}\mid k\ \rm{even}\}$ and $A_2=\{\{2^k,2^{k+1}\}\mid k\ \rm{odd}\}$. Since $-1$ is an even power of $2$ in $G_p$, we conclude that $\Delta A_1$ consists of only even powers of 2 in $\abr 2_p$, and $\Delta A_2$ consists of only odd powers of 2. That means that neither $\Delta A_1$ nor $\Delta A_2$ comprise $\langle \underline{2}\rangle_p$. The obtained contradiction completes the proof of non-existence of a cardioidal starter of order $p$, prime $p\notin C_2$.
\end{enumerate}
Hence, if a prime $p\notin C_2$ divides $n$, a cardioidal starter of order $n$ does not exist.
\end{proof}

\begin{ex}The partition technique proposed in Theorem \ref{major theorem}.
\begin{enumerate}
\item n=$11^4$. \\
$\mathbb{Z}_{11^4}^*=11^3\underline{G}_{11}\cup11^2\underline{G}_{11^2}\cup11\underline{G}_{11^3}\cup \underline{G}_{11^4}$ .
\item n=$11\cdot19\cdot43$.\\
$\mathbb{Z}_{11\cdot19\cdot43}^*=11\cdot19\underline{G}_{43}\cup11\cdot43\underline{G}_{19}\cup19\cdot43\underline{G}_{11}\cup11\underline{G}_{19\cdot43}\cup19\underline{G}_{11\cdot43}\cup43\underline{G}_{11\cdot19}\cup \underline{G}_{11\cdot19\cdot43}$ .
\end{enumerate}
\end{ex}
Finally, we give the necessary and sufficient conditions for the existence of strong cardioidal starters. 
\begin{theorem}\label{final theorem}A strong cardioidal starter of order $n$ exists iff $n\in\overline{C_2\setminus\{3\}}$.

\end{theorem}
\begin{proof} Let $n\in\overline{C_2\setminus\{3\}}$, that is, $3\nmid n,\ n\in\overline{C_2}$. By Theorem \ref{major theorem}, there exists a cardioidal starter of order $n$. By Lemma \ref{basic statement}, it is strong\footnote{by Theorem \ref{skew}, it is skew.}.

Conversely, if $n\notin\overline{C_2}$, then a cardioidal starter of order $n$ does not exist by Theorem \ref{major theorem}.  If $3|n,\ n\in\overline{C_2}$, then a cardioidal starter of order $n$ is not strong by Lemma \ref{basic statement}. 
\end{proof}
\begin{ex}The construction of strong cardioidal starters of composite orders in details.
\begin{enumerate}
\item $n=121=11^2. \\ \mathbb{Z}^*_{121}=11\underline{G}_{11}\cup \underline{G}_{121}.$\\
\noindent The pairs from $11\underline{G}_{11}:\ \{11,22\},\{44,88\},\{55,110\},\{99,77\},\{33,66\}$;\\
The pairs from $\langle \underline{2}\rangle_{121}= \underline{G}_{121}:\ \{1,2\},\{4,8\},\{16,32\},\{64,7\},\{14,28\},\{56,112\},\\ 
\{103,85\},\{49,98\},\{75,29\},\{58,116\},\{111,101\},\{81,41\},\{82,43\},\{86,51\},\{102,83\},\\ \{45,90\},\{59,118\},
\{115,109\},\{97,73\},\{25,50\},\{100,79\},\{37,74\},\{27,54\},\{108,95\},\\
\{34,68\},\{15,30\},\{60,120\},\{119,117\},\{113,105\},\{89,57\},\{114,107\},\{93,65\},\\
\{69,17\},\{9,18\},\{36,72\},\{23,46\},\{92,63\},\{5,10\},\{20,40\},\{80,39\},\{78,35\},\{70,19\},\\
\{38,76\},\{31,62\},\{3,6\},\{12,24\},\{48,96\},\{71,21\},\{42,84\},\{47,94\},\{67,13\},\{26,52\},\\
\{104,87\},\{53,106\},\{81,61\}.$

\item $n=209=11\cdot19.\\ \mathbb{Z}^*_{209}=11\underline{G}_{19}\cup19\underline{G}_{11}\cup\underline{G}_{209};\ \underline{G}_{209}=\langle \underline{2}\rangle_{209}\cup3\langle \underline{2}\rangle_{209}.$\\
The pairs from $11\underline{G}_{19}:\ \{11,22\},\{44,88\},\{176,143\},\{77,154\},\{99,198\},\{187,165\},\\
\{121,33\},\{66,132\},\{55,110\}$;\\
The pairs from $19\underline{G}_{11}:\ \{19,38\},\{76,152\},\{95,190\},\{171,133\},\{57,114\}$;\\
The pairs from $\langle \underline{2}\rangle_{209}:\ \{1,2\},\{4,8\},\{16,32\},\{64,128\},\{47,94\},\{188,167\},\{125,41\},\\
\{82,164\},\{119,29\},\{58,116\},\{23,46\},\{92,184\},\{159,109\},\{9,18\},\{36,72\},\{144,79\},\\
\{158,107\},\{5,10\},\{20,40\},\{80,160\},\{111,13\},\{26,52\},\{104,208\},\{207,205\},\\
\{201,193\},\{177,145\},\{81,162\},\{115,21\},\{42,84\},\{168,127\},\{45,90\},\{180,151\},\\
\{93,186\},\{163,117\},\{25,50\},\{100,200\},\{191,173\},\{137,65\},\{130,51\},\{102,204\},\\
\{199,189\},\{169,129\},\{49,98\},\{196,183\},\{157,105\}$;\\
 The pairs from $3\langle \underline{2}\rangle_{209}:\ \{3,6\},\{12,24\},\{48,96\},\{192,175\},\{141,73\},\{146,83\},\\ \{166,123\},\{37,74\},\{148,87\},\{174,139\},\{69,138\},\{67,134\},\{59,118\},\{27,54\},\\
\{108,7\},\{14,28\},\{56,112\},\{15,30\},\{60,120\},\{31,62\},\{124,39\},\{78,156\},\{103,206\},\\
\{203,197\},\{86,172\},\{185,161\},\{113,17\},\{34,68\},\{136,63\},\{126,43\},\{135,61\},\\
\{122,35\},\{70,140\},\{71,142\},\{75,150\},\{91,182\},\{155,101\},\{202,195\},\{181,153\},\\
\{97,194\},\{179,149\},\{89,178\},\{147,85\},\{170,131\},\{53,106\}.$
\end{enumerate}
\end{ex}
\section{Cardioidal starters of order $n$ when $3\mid n$}\label{order divisible by 3}
The construction presented in Theorem \ref{major theorem} and in Lemma \ref{group of units}, if $3|n,\ n\in\overline{C_2}$, produces a cardioidal $($and hence, Skolem$)$ starter of order $n$ which, by Lemma \ref{basic statement}, is not strong. However, it is a convenient way to generate Skolem sequences of order $(n-1)/2$.

\begin{ex}
Let us construct a cardioidal starter of order $n$ divisible by $3$.\\
 $n=33=3\cdot 11$.\\
$\ZZ^*_{33}=3\underline{G}_{11}\cup11\underline{G}_3\cup\underline{G}_{33}.\ \underline{G}_{33}=\langle \underline{2}\rangle_{33}\cup5\langle \underline{2}\rangle_{33}$.\\
The pairs from $3\underline{G}_{11}:\ \{3,6\},\{12,24\},\{15,30\},\{27,21\},\{9,18\};$\\
\noindent the pair from $11\underline{G}_3:\ \{11,22\};$\\
\noindent  the pairs from $\langle \underline{2}\rangle_{33}:\ \{1,2\},\{4,8\},\{16,32\},\{31,29\},\{25,17\}$;\\
\noindent  the pairs from $5\langle \underline{2}\rangle_{33}:\ \{5,10\},\{20,7\},\{14,28\},\{23,13\},\{26,19\}$.\\
The corresponding Skolem sequence is 
$(1,1,3,4,5,3,13,4,9,5,11,12,10,14,15,16,8,9,7,13,\\6,11,10,12,8,7,6,14,2,15,2,16)$.

The starter is cardioidal but not strong. In the set of the $16$ corresponding chords, the chord $[11,22]$ is horizontal; the other 15 chords split into five triples of parallel chords, for example, $[1,2]\parallel [23,13]\parallel[12,24]$ since $1+2\equiv23+13\equiv12+24\pmod{33}$. 
\end{ex}

As strong Skolem starters of orders other than mentioned in Theorem \ref{final theorem} can not be cardioidal, their construction should be based on different principles. The results of this paper may be applicable in constructing strong Skolem starters of order $3n$, using a strong cardioidal starter of order $n$. For instance, the problem of the existence of a strong Skolem starter of order $3n$, where $n\in\overline{C_2\setminus\{3\}}$, may be reduced to finding a Langford sequence\footnote{A Langford sequence of defect $d$ and length $m$, $\mathcal {L}^d_{m}$, is a generalized Skolem sequence with $2m$ symbols, where each $j\in\{d,d+1,...,d+m-1\}$ appears in exactly two positions, $a_j$ and $b_j$, so that $b_j-a_j=j$ \cite{b04}.} $\mathcal {L}^{(n+1)/2}_{n}$, which mathches to a Skolem sequence of order $(n-1)/2$ whose constraction follows from Theorem \ref{major theorem} and Lemma \ref{group of units}. By matching, we mean that a concatenation of the two sequences results in a Skolem sequence of order $(3n-1)/2$, producing a strong Skolem starter of order $3n$.
The following example will make it clear:
\begin{ex}
 Consider the Skolem sequence, generating a strong Skolem starter of order $33:\ (11,12,13,14,15,16,6,7,8,9,10,11,6,12,7,13,8,14,9,15,10,16,5,2,4,2,3,5,4,3,1,1)$. It can be noted that this sequence is obtained by concatenation of a Langford sequence $\mathcal {L}^6_{11}$ and a Skolem sequence of order 5, corresponding to a strong cardioidal starter of order $11:\\  (11,12,13,14,15,16,6,7,8,9,10,11,6,12,7,13,8,14,9,15,10,16)|(5,2,4,2,3,5,4,3,1,1)$.
\end{ex}
The question of the existence and effective construction of a Langford sequence matching a given Skolem sequence, generated by a strong cardioidal starter, is an open problem.
\section{Conclusions and further research}\label{conclusions}
In this paper, we introduced and studied new combinatorial objects, namely, strong cardioidal starters in order to address the question of the existence of strong Skolem starters, the former comprising a proper subset of the latter. While this question remains open, we managed to partly confirm Shalaby's Conjecture stated in 1991, by explicitly constructing infinite families of strong cardioidal starters. Theorem \ref{final theorem} fully describes the infinite subset of natural numbers admitting strong cardioidal (and hence, strong Skolem) starters of the corresponding orders. Due to our discovery of infinite families of strong Skolem starters, there arise further questions over the existence of strong starters, generated by extended Skolem sequences (Rosa, hooked Skolem and other Skolem-type sequences \cite{b014}).

In addition, in the proofs of Theorem \ref{major theorem} and Lemma \ref{group of units}, we proposed a new way to generate cardioidal (but not necessarily strong) starters for all possible orders. These starters give rise to corresponding Skolem sequences, which are by themselves valuable objects in the theory of mathematical design and its applications, for example, in construction of constant-weight design codes \cite{b03} and various Skolem-type rectangles \cite{b017}. \\

\chapter{\Large\bf{Appendix A}}\label{appendix A}

We present below examples of Skolem sequences of orders $q< 30$ that yeild strong Skolem starters of orders $2q+1$. All the sequences, except the last one, were constructed in \cite{b09}. \\
$q=5$\\
$(5,2,4,2,3,5,4,3,1,1)$;\\
$q=8$\\
$(5,6,7,8,2,5,3,6,4,7,3,8,4,3,1,1)$;\\
$q=9$\\
$(8,9,3,4,7,3,6,4,8,5,9,7,6,2,5,2,1,1)$;\\
$q=12$\\
$(4,5,11,8,4,10,5,7,9,12,2,8,2,11,7,10,6,9,1,1,3,12,6,3)$;\\
$q=13$\\
$(8,9,10,11,12,13,7,4,8,6,9,4,10,7,11,6,12,5,13,2,3,2,5,3,1,1)$;\\
$q=16$\\
$(11,12,13,14,15,16,6,7,8,9,10,11,6,12,7,13,8,14,9,15,10,16,5,2,4,2,3,5,4,3,1,1)$;\\
$q=17$\\
$(11,12,13,14,15,16,17,2,6,2,8,11,10,12,6,13,9,14,8,15,7,16,10,17,5,9,4,7,3,5,4,3,1,1)$;\\
$q=20$\\
$(13,14,15,16,17,18,19,20,8,3,12,7,3,13,10,14,8,15,7,16,11,17,12,18,10,19,6,20,5,9,4,\\11,6,5,4,1,1,2,9,2)$;\\
$q=21$\\
$(13,14,15,16,17,18,19,20,21,11,12,5,6,13,10,14,5,15,6,16,11,17,12,18,10,19,9,20,8,21,\\3,7,2,3,2,9,8,4,7,1,1,4)$;\\
$q=24$\\
$(16,17,18,19,20,21,22,23,24,6,7,8,11,12,15,6,16,7,17,8,18,14,19,11,20,12,21,10,22,\\15,23,5,24,9,13,14,5,10,2,4,2,3,9,4,3,1,1,13)$;\\
$q=25$\\
$(16,17,18,19,20,21,22,23,24,5,9,25,3,10,5,3,16,4,17,9,18,4,18,10,20,15,21,8,22,14,\\23,7,24,11,13,8,25,12,7,6,15,1,1,14,11,6,2,16,2,12)$;\\
$q=28$\\
$(18,19,20,21,22,23,24,25,26,27,28,6,7,3,15,8,3,6,18,7,19,16,20,8,21,17,22,14,23,15,\\24,9,25,13,26,12,27,16,28,11,9,14,17,5,10,4,13,12,5,4,11,1,1,2,10,2)$;\\
$q=29$\\
$(1,1,3,4,5,3,7,4,9,5,24,12,23,7,15,16,17,9,19,20,21,22,18,12,25,26,27,28,29,15,14,\\16,13,17,24,23,11,19,10,20,18,21,8,22,14,13,6,11,10,25,8,26,6,27,2,28,2,29)$.\\

\chapter{\Large\bf{Appendix B}}\label{appendix B}

A strong cardioidal starter of order 281. $\underline{G}_{281}=\langle \underline{2}\rangle_{281}\cup3\langle \underline{2}\rangle_{281}\cup5\langle \underline{2}\rangle_{281}\cup15\langle \underline{2}\rangle_{281}$.\\
 The pairs from $\langle \underline{2}\rangle_{281}:$\\
$\{1,2\},\{4,8\},\{16,32\},\{64,128\},\{256,231\},\{181,81\},\{162,43\},\{86,172\},\{63,126\},\\
\{252,223\},\{165,49\},\{98,196\},\{111,222\},\{163,45\},\{90,180\},\{79,158\},\{35,70\},\\
\{140,280\},\{279,277\},\{271,265\},\{249,217\},\{153,25\},\{50,100\},\{200,119\},\{238,195\},\\
\{109,218\},\{155,29\},\{58,116\},\{232,183\},\{85,170\},\{59,118\},\{236,191\},\{101,202\},\\
\{123,246\},\{211,141\}$;\\
the pairs from $3\langle \underline{2}\rangle_{281}:$\\
$\{3,6\},\{12,24\},\{48,96\},\{192,103\},\{206,131\},\{262,243\},\{205,129\},\{258,235\},\{189,97\},\\
\{194,107\},\{214,147\},\{13,26\},\{52,104\},\{208,135\},\{270,258\},\{237,193\},\{105,210\},\\
\{139,278\},\{275,269\},\{257,233\},\{185,89\},\{178,75\},\{150,19\},\{38,76\},\{152,23\},\\
\{46,92\},\{184,87\},\{174,67\},\{134,268\},\{255,229\},\{177,73\},\{146,11\},\{22,44\},\{88,176\},\\
\{71,142\}$;\\
the pairs from $5\langle\underline{2}\rangle_{281}:$\\
$\{5,10\},\{20,40\},\{80,160\},\{39,78\},\{156,31\},\{62,124\},\{248,215\},\{149,17\},\{34,68\},\\
\{136,272\},\{263,245\},\{209,137\},\{274,267\},\{253,225\},\{169,57\},\{114,228\},\{175,69\},\\
\{138,276\},\{271,261\},\{241,201\},\{121,242\},\{203,125\},\{250,219\},\{157,33\},\{66,132\},\\
\{264,247\},\{213,145\},\{9,18\},\{36,72\},\{144,7\},\{14,28\},\{56,112\},\{224,167\},\{53,106\},\\
\{212,143\}$;\\
the pairs from $15\langle\underline{2}\rangle_{281}:$\\
$\{15,30\},\{60,120\},\{240,199\},\{117,234\},\{187,93\},\{186,91\},\{182,83\},\{166,51\},\\
\{102,204\},\{127,254\},\{227,173\},\{65,130\},\{260,239\},\{197,113\},\{226,171\},\{61,122\},\\
\{244,207\},\{133,266\},\{251,221\},\{161,41\},\{82,164\},\{47,94\},\{188,95\},\{190,99\},\\
\{198,115\},\{230,179\},\{77,154\},\{27,54\},\{108,216\},\{151,21\},\{42,84\},\{168,55\},\\
\{110,220\},\{159,37\},\{74,148\}.$


\begin{thebibliography}{}

\bibitem{b012}  Beardon, A.F.  and Beardon, L.A. Circles, chords and epicycloids. The Mathematical Gazette, 73 no.465 (1989), 192-197 
\bibitem{b019}  Brauer, A. A note on a number theoretical paper of Sierpinski. Proc. Amer. Math. Soc., 11 (1960), 406--409
\bibitem{b014} Colbourn, C.J. and Dinitz, H. Handbook of Combinatorial Design. CRC Press., 2007
\bibitem{b06} Dinitz, J.H. and Stinson, D.R. Contemporary design theory: a collection of Surveys, 1992
\bibitem{b07} Dinitz, J.H. and Stinson, D.R. A fast algoritm for finding strong starters. SIAM J. Alg. Disc. Math., 2:1 (1981), 50-56
\bibitem{b05} Horton, J.D. Orthogonal starters in finite abelian groups. Discrete Mathematics, 79 (1989/90), 265--278 
\bibitem{b015} Ireland, K., Rosen, M., A classical introduction to modern number theory. Springer-Verlag, 2nd ed., 1990

\bibitem{b03} Lan, L., Chang, Y. and Wang, L. Construction of cyclic quaternary constant-weight codes of weight three and distance four. Designs, Codes and Cryptography, 86:5  (2018), 1063-1083

\bibitem{b08} Lawrence, J. D. A Catalog of Special Plane Curves. New York: Dover, 1972
\bibitem{b017} Linek, V., Jiang, Z. Extended Langford Sequences with Small Defects. Journal of Combinatorial  Theory,  Series  A, 84 (1998), 38-54  
\bibitem{b013} Linek, V., Mor, S., Shalaby, N. Skolem and Rosa rectangles and related designs. Discrete Mathematics, 331 (2014), 53-73

\bibitem{b01} Mullin, R.C. and Nemeth, E. An existence theorem for Room squares. Canad. Math. Bull., 12 (1969),  493--497

\bibitem{b016} Mullin, R.C., Stanton, R. G. Construction of Room Squares. Ann. Math. Statist., 39:5 (1968),  1540--1548

\bibitem{b09} Shalaby, N. Skolem sequences: generalizations and applications. Thesis (PhD). McMaster University (Canada), 1991

\bibitem{b04} Simpson J.E. Langford sequences: perfect and hooked. Discrete Mathematics, 44 (1983), 97--104 

\bibitem{b02} Skolem, T. On certain distributions of integers in pairs with given differences. Mathematica Scandinavica, 5 (1957), 57-68
\end{thebibliography}
\end{document}